\documentclass[12pt]{amsart}
\usepackage[a4paper, top=2.8cm, left=2cm, right=2cm, bottom=2cm]{geometry}
\usepackage{graphicx}
\usepackage{amsfonts}
\usepackage{amsmath}
\usepackage{amssymb}
\usepackage{amsthm}
\usepackage{newlfont}
\usepackage{yhmath}
\usepackage{color}

\newtheorem{lemma}{Lemma}[section]

\newtheorem{thm}{Theorem}[section]

\newcommand{\p}{\mathbb{P}}

\newcommand{\ii}{\mathbf{i}}
\allowdisplaybreaks

\begin{document}

\title{Integral Representation of Probabilities\\
 in Kingman Coalescent}

%    Remove any unused author tags.

%    author one information
\author{Youzhou Zhou}
\address{Department of Mathematical Science\\
Xi'an Jiaotong-Liverpool University \\
111 Renai Road\\
Suzhou, China 215 123}
%\curraddr{}
\email{youzhou.zhou@xjtlu.edu.cn}
\thanks{This research is supported by NSFC: 11701570.}

%\subjclass[2000]{Primary }
%    For articles to be published after 1 January 2010, you may use
%    the following version:


%\keywords{Homozygosity, Phase Transition, Mean Field Spin Glass, Critical Temperature}
\begin{abstract}
 Kingman Coalescent was first proposed by Kingman~\cite{MR671034} in population genetics to describe population's genealogical structure. Now it becomes a benchmark model for coalescent process. Extensive studies have been conducted on Kingman coalescent. In particular, its explicit finite time distribution was obtained by Tavar\'e~\cite{MR770050}. However, very few people use this explicit distribution to do analysis for it is an intractable infinite series.  In this article, we are going to establish a complex integral representation for the finite time distribution, then we use steepest descent method to analyze this integral representation to obtain local central limit theorem at small time regime.  
 \end{abstract}
 
 \date{\today}

\dedicatory{}

\maketitle
\allowdisplaybreaks
\section{Introduction}
Kingman coalescent is a first coalescent model in population genetics. It was first proposed by Kingman~\cite{MR671034} to describe the genealogical structure of a sample. It eventually finds its extensive applications in biology.  Also there are many generalizations of Kingman coalescent such as $\Lambda-$coalescent~\cite{MR1742892}. For recent development of coalescent processes, one can refer to \cite{MR2574323}. All these coalescent processes have a nice dual correspondence with evolution dynamics. Because if we observe the population forward in time we get evolution process, if, however, we look the population backward in time, we end up with coalescent.  Coalescing structure is ubiquitous in nature, one can refer to \cite{MR1673235} for a nice review on mathematical coalescent models in physics and chemistry. Coalescent process can be treated in different mathematical setting as well. For instance, Kingman coalescent is a Markov chain, but it can also be regarded as a random metric space by Evans in~\cite{MR1765005}. In this article, we will only regard Kingman coalescent as a Markov chain. 

To describe genealogical structure, one can use an equivalent relation defined according to whether two individuals share common parents in previous generation. Each equivalent relation will generate a partition of a sample. Therefore, it is natural to use partitions of this sample to represent its genealogical structures. Let $[n]=\{1,2,\cdots,n\}$ be a sample of size $n$. Denote $A_{n}$ to be the set of all partitions of this sample. Kingman coalescent $X_{t}^{n}$ is defined to be a continuous time Markov chain valued in $A_{n}$. If we only consider the number of blocks of partition, then we will get an integer-valued pure death process, called block-counting process and denoted as $D_{t}^{n}$. The transition rate of $D_{t}^{n}$ is 
$$
q_{n,m}=\begin{cases}
\frac{n(n+\theta-1)}{2},& m=n-1\\
-\frac{n(n+\theta-1)}{2}&  m=n\\
0& m\neq n,n-1
\end{cases}
$$
where $\theta>0$ is mutation rate. 
As $n\to\infty$, $D_{t}^{n}$ will converges to a pure death chain $D_{t}$ whose initial state is $\infty$. The finite time probability mass function of $D(t)$ is $d_{n}^{\theta}(t)=\p(D_{t}=n)$. These probabilities sever as convex coefficients of  the explicit transition function of Fleming-Viot process with parent-independent mutation \cite{MR1235429} and the ininitely-many-neutral-alleles model \cite{MR3336858}. In 1984, Tavar\'e obtained the following explicit expression of $d_{n}^{\theta}(t)$ in \cite{MR770050} 
\begin{align*}
d_{0}^{\theta}(t)=&1-\sum_{k=1}^{\infty}\frac{2k-1+\theta}{k!}(-1)^{k-1}\theta_{(k-1)}e^{-\lambda_{k}t}\\
d_{n}^{\theta}(t)=&\sum_{k=1}^{\infty}\frac{2k-1+\theta}{k!}(-1)^{k-n}\binom{k}{n}(n+\theta)_{(k-1)}e^{-\lambda_{k}t},n\geq1.\\
\lambda_{k}&=\frac{k(k+\theta-1)}{2},\theta_{(k)}=\frac{\Gamma(\theta+k)}{\Gamma(\theta)},\quad k\geq0.
\end{align*}
It has been 30 years, people only use this expression to derive ergodic inequality of Fleming-Viot process with parent-independent mutation and the ininitely-many-neutral-alleles model. 

One distinctive feature of Kingman coalescent is that it has an entrance boundary $\infty$, also termed as coming down from infinity. Therefore this process will jump to finite state immediately, and many details are hided in the small time regime. There are many studies on the small time asymptotic of Kingman coalescent~\cite{MR3342667},\cite{MR3339865},\cite{MR770713},\cite{MR3304413}. However, none of them have ever used the explicit expression of $d_{n}^{\theta}(t)$. One possible reason is that the expression of $d_{n}^{\theta}(t)$ is intractable series, and some probability tools, such as martingale, are already powerful enough to handle problems about Kingman coalescent. But if one would like to derive some fine asymptotic result, it is usually more hopeful to work with the explicit distribution~$d_{n}^{\theta}(t)$. 

 In this article, we are going to find an integral representation for probabilities $d_{n}^{\theta}(t)$. Integral representation is usually more suitable for asymptotic analysis. The idea to establish integral representation is to replace some combinatorial terms and $e^{-\lambda_{k}}$ by integrals. Cauchy integral formula and Fourier transform are tailor-made tool for this job. Then one will end up with an integral whose integrand is essentially a geometric series. Therefore, one can easily obtained a simplified integrals by summing up the geometric series. Sometimes, if possible, one also needs to do carefully residue computation to simplify the integral further. Once we have an integral representation, we will do some contour deformation and then use steepest descent method to obtain the asymptotic behavior. As an application, we obtain local central limit theorem. In \cite{MR3342667} and \cite{MR3339865}, Limic and Talarczyk obtained functional fluctuation theorem at small time regime. However, local central limit theorem is new. In \cite{MR3304413}, Depperschmidt, Pfaffelhuber and Scheuringer discussed small time large deviations for Kingman coalescent. We believe the integral representation in this article can also be used to get both large deviations and moderate deviations. 
 
 This article is planned as follows: In section 2, we will discuss the integral representation in Theorem \ref{asm} and its proof. This integral representation is essentially the same as (2.16) in \cite{2006G}.  In section 3, we will present an alternative integral representation in Theorem \ref{arep} which is more suitable for asymptotic analysis, then we use steepest descent method to obtain local central limit theorem. This alternative integral representation is brand new.  Proofs of some lemmas are left in section 4. Last, we will give some remarks in section 5. 
 
 In this article, notation ``$\sim$" will reappear many times. We say $a(t)\sim b(t)$ if $\lim_{t\to0}\frac{a(t)}{b(t)}=1$.  Notation $M_{1},M_{2}$ for constants will also show up a few times. Whenever they appear, we just regard them as a constant independent of $t$. Also $\ii$ is a reserved word for imaginary unit and $\ii^2=-1$.

\section{Integral Representation for $d_{n}^{\theta}(t)$}

 In this section, we are going to derive some integral representations for probabilities $d_{n}^{\theta}(t)$ of Kingman coalescent. Generally speaking, integral representation is usually more tractable than series expression if one want to directly use $d_{n}^{\theta}(t)$ to study small time asymptotic of Kingman coalescent.

In series expression of $d_{n}^{\theta}(t)$, there are combinatorial coefficients and $e^{-\lambda_{k}t}$. One can express them in terms of complex integrals with specifically chosen contours, then the series expansion becomes complex integral whose integrand is essentially a geometric series. Therefore, one can easily sum up this to get a simplified integrand. To rewrite combinatorial coefficients as contour integrals, one only needs to use Cauchy integral formula. Because $e^{-\lambda_{k}t}$ is of gaussian type, one can use Fourier transform to get its integral form. Last, one also need to do careful residue calculation to get the following integral representation.

\begin{thm}\label{asm}
For $n\geq0$, $d_{n}^{\theta}(t)$ has the following integral representation
\begin{align*}
d_{n}^{\theta}(t)=\binom{2n-1+\theta}{n}\frac{1}{2^{2n-1+\theta}}e^{\frac{(1-\theta)^2t}{8}}
\frac{1}{\sqrt{2\pi t}}\int_{-\infty-\ii A}^{\infty-\ii A}
e^{-\frac{w^2}{2t}}\frac{\ii\sin\frac{w}{2}}{(\cos\frac{w}{2})^{2n+\theta}}dw
\end{align*}
where $A>0$ and $(\cos\frac{w}{2})^{2n+\theta}$ is defined to be principal branch of power function $z^{2n+\theta}$. 
\end{thm}

\begin{proof}
First, by simple algebraic calculation, we have  
\begin{align*}
&\frac{2k-1+\theta}{k!}(-1)^{k-n}\binom{k}{n}(n+\theta)_{(k-1)}=(-1)^{k-n}\left[\frac{(n+\theta)_{(k)}}{k!}\binom{k}{n}+\frac{(n+\theta)_{(k-1)}}{(k-1)!}\binom{k-1}{n}\right]\\
\end{align*}
and also 
$$
\lambda_{k}=\frac{1}{2}(k^2+(\theta-1)k)=\frac{1}{2}\left[k+\frac{\theta-1}{2}\right]^2-\frac{(\theta-1)^2}{8}.
$$
By Fourier transform, we have 
\begin{align*}
e^{-\lambda_{k}t}=\exp\left\{\frac{(\theta-1)^2t}{8}\right\}\frac{1}{\sqrt{2\pi}}\int_{-\infty-\ii a}^{\infty-\ii a}
\exp\left\{-\frac{\xi^2}{2}-\ii\xi(\frac{\theta-1}{2})\sqrt{t}\right\}e^{-\ii\xi\sqrt{t} k}d\xi
\end{align*}
where $a>\frac{\log\frac{3}{2}}{\sqrt{t}}$. Moreover, due to Cauchy integral formula, one has
\begin{align*}
&\frac{(n+\theta)_{(k)}}{k!}=\frac{1}{2\pi \ii}\int_{C_{1}}\frac{\xi_{1}^{n+\theta+k-1}}{(\xi_{1}-1)^{k+1}}d\xi_{1}\\
&\binom{k}{n}=\frac{1}{2\pi \ii}\int_{C_{2}}\frac{\xi_{2}^{k}}{(\xi_{2}-1)^{n+1}}d\xi_{2}
\end{align*}
where $C_{1}$ is a small circle centered at $1$ with radius $r>\frac{3}{2e^{\sqrt{t}a}-3}$ and $C_{2}$ is a circle with center $1$ and radius $\frac{1}{2}$. These contours are chosen to guarantee the uniform convergence of upcoming series. 

Replacing combinatorial coefficients and $e^{-\lambda_{k}t}$ by their integral forms in $d_{n}^{\theta}(t)$, one can have 
\begin{align*}
d_{n}^{\theta}(t)
=&\sum_{k=n}^{\infty}(-1)^{k-n}\frac{1}{(2\pi \ii)^2}\Big[\int_{C_{1}}\int_{C_{2}}\frac{\xi_{1}^{n+\theta+k-1}}{(\xi_{1}-1)^{k+1}}d\xi_{1} \frac{\xi_{2}^{k}}{(\xi_{2}-1)^{n+1}}d\xi_{2}+\\
&\int_{C_{1}}\int_{C_{2}}\frac{\xi_{1}^{n+\theta+k-2}}{(\xi_{1}-1)^{k}}d\xi_{1}\frac{\xi_{2}^{k-1}}{(\xi_{2}-1)^{n+1}}d\xi_{2}\Big]\exp\Big\{\frac{(\theta-1)^2t}{8}\Big\}\\
&\frac{1}{\sqrt{2\pi}}\int_{-\infty-\ii a}^{\infty-\ii a }
\exp\Big\{-\frac{\xi_{3}^2}{2}-\ii\xi_{3}(\frac{\theta-1}{2})\sqrt{t}\Big\}e^{-\ii\xi_{3}\sqrt{t} k}d\xi_{3}\\
=&\frac{1}{\sqrt{2\pi}(2\pi \ii)^2}\exp\Big\{\frac{(\theta-1)^2t}{8}\Big\}\sum_{k=n}^{\infty}\int_{C_{1}}\int_{C_{2}}\int_{-\infty-\ii a}^{\infty-\ii a}\\
&\exp\Big\{-\frac{\xi_{3}^2}{2}-\ii\xi_{3}(\frac{2n+\theta-1}{2})\sqrt{t}\Big\}\Big[-\frac{\xi_{1}\xi_{2}}{e^{\ii\xi_{3}\sqrt{t}}(\xi_{1}-1)}\Big]^{k-n}\\
&\frac{\xi_{1}^{2n-2+\theta}\xi_{2}^{n-1}(\xi_{1}\xi_{2}+\xi_{1}-1)}{(\xi_{1}-1)^{n+1}(\xi_{2}-1)^{n+1}}d\xi_{1}d\xi_{2}
d\xi_{3}
\end{align*}
The specifically chosen contours guarantee that 
$$
\left|\frac{\xi_{1}\xi_{2}}{e^{\ii\xi_{3}\sqrt{t}}(\xi_{1}-1)}\right|<1.
$$
Thus, it is safe to switch summation and integration. After summing up the geometric series, one has 
\begin{align*}
d_{n}^{\theta}(t)
=&\frac{1}{\sqrt{2\pi}(2\pi \ii)^2}\exp\Big\{\frac{(\theta-1)^2t}{8}\Big\}\int_{C_{1}}\int_{C_{2}}\int_{-\infty-\ii a}^{\infty-\ii a}\\
&\exp\Big\{-\frac{\xi_{3}^2}{2}-\ii\xi_{3}(\frac{2n+\theta-3}{2})\sqrt{t}\Big\}\frac{(\xi_{1}\xi_{2}+\xi_{1}-1)}{\xi_{1}\xi_{2}+e^{\ii\xi_{3}\sqrt{t}}(\xi_{1}-1)}\\
&\frac{\xi_{1}^{2n-2+\theta}\xi_{2}^{n-1}}{(\xi_{1}-1)^{n}(\xi_{2}-1)^{n+1}}d\xi_{1}d\xi_{2}
d\xi_{3}\\
=&\frac{1}{\sqrt{2\pi}(2\pi \ii)^2}\exp\Big\{\frac{(\theta-1)^2t}{8}\Big\}\int_{-\infty-\ii a}^{\infty-\ii a}\exp\Big\{-\frac{\xi_{3}^2}{2}-\ii\xi_{3}(\frac{2n+\theta-3}{2})\sqrt{t}\Big\}d\xi_{3}\\
&\int_{C_{2}}\frac{\xi_{2}^{n-1}}{(\xi_{2}+e^{\ii\sqrt{t}\xi_{3}})(\xi_{2}-1)^{n+1}}d\xi_{2}\int_{C_{1}}\frac{(\xi_{1}\xi_{2}+\xi_{1}-1)\xi_{1}^{2n-2+\theta}}{(\xi_{1}-1)^{n}(\xi_{1}-\frac{e^{\ii\sqrt{t}\xi_{3}}}{\xi_{2}+e^{\ii\sqrt{t}\xi_{3}}})}d\xi_{1}
\end{align*}
Now we first integrate out $\xi_{1}$, then 
$$
\frac{1}{2\pi\ii}\int_{C_{1}}\frac{(\xi_{1}\xi_{2}+\xi_{1}-1)\xi_{1}^{2n-2+\theta}}{(\xi_{1}-1)^{n}(\xi_{1}-\frac{e^{\ii\sqrt{t}\xi_{3}}}{\xi_{2}+e^{\ii\sqrt{t}\xi_{3}}})}d\xi_{1}=R_{1}(\xi_{2},\xi_{3})+R_{2}(\xi_{2},\xi_{3})
$$
where $R_{1}(\xi_{2},\xi_{3}), R_{2}(\xi_{2},\xi_{3})$ are residues at $1$ and $\frac{e^{\ii\sqrt{t}\xi_{3}}}{\xi_{2}+e^{\ii\sqrt{t}\xi_{3}}}$ and 
\begin{align*}
R_{1}(\xi_{2},\xi_{3})=&\sum_{l=0}^{n-1}(-1)^{n-l-1}(n-l-1)!\binom{n-1}{l}\left[\frac{(2n-1+\theta)!}{(2n-l-1+\theta)!}(\xi_{2}+1)
-\frac{(2n-2+\theta)!}{(2n-l-2+\theta)!}\right]\\
&(\frac{\xi_{2}+e^{\ii\xi_{3}\sqrt{t}}}{\xi_{2}})^{n-l}\\
R_{2}(\xi_{2},\xi_{3})=&(-1)^n\frac{(e^{\ii\xi_{3}\sqrt{t}}-1)e^{\ii\xi_{3}\sqrt{t}(2n-2+\theta)}}{\xi_{2}^{n-1}(\xi_{2}+e^{\ii\xi_{3}\sqrt{t}})^{n-1+\theta}}.
\end{align*}
Then 
\begin{align*}
d_{n}^{\theta}(t)
=&\frac{1}{\sqrt{2\pi}(2\pi \ii)^2}\exp\Big\{\frac{(\theta-1)^2t}{8}\Big\}\int_{-\infty-\ii a}^{\infty-\ii a}\exp\Big\{-\frac{\xi_{3}^2}{2}-\ii\xi_{3}(\frac{2n+\theta-3}{2})\sqrt{t}\Big\}d\xi_{3}\\
&\int_{C_{2}}\frac{\xi_{2}^{n-1}}{(\xi_{2}+e^{\ii\sqrt{t}\xi_{3}})(\xi_{2}-1)^{n+1}}[R_{1}(\xi_{2},\xi_{3})+R_{2}(\xi_{2},\xi_{3})]d\xi_{2}
\end{align*}

One can show that 
$$
\frac{1}{2\pi\ii}\int_{C_{2}}\frac{\xi_{2}^{n-1}}{(\xi_{2}+e^{\ii\sqrt{t}\xi_{3}})(\xi_{2}-1)^{n+1}}R_{1}(\xi_{2},\xi_{3})d\xi_{2}=0
$$
because the maximum degree of numerator is $n-1$ and the degree of denominator is $n+1$. Moreover, 
\begin{align*}
&\frac{1}{2\pi \ii}\int_{C_{2}}\frac{\xi_{2}^{n-1}}{(\xi_{2}+e^{\ii\sqrt{t}\xi_{3}})(\xi_{2}-1)^{n+1}}R_{2}(\xi_{2},\xi_{3})d\xi_{2}\\=&(-1)^n(e^{\ii\xi_{3}\sqrt{t}}-1)e^{i\xi_{3}\sqrt{t}(2n-2+\theta)}\frac{1}{2\pi\ii}\int_{C_{2}}\frac{1}{(\xi_{2}-1)^{n+1}(\xi_{2}+e^{\ii\xi_{3}\sqrt{t}})^{n+\theta}}d\xi_{2}\\
=&(e^{\ii\xi_{3}\sqrt{t}}-1)e^{i\xi_{3}\sqrt{t}(2n-2+\theta)}\binom{2n-1+\theta}{n}(1+e^{\ii\xi_{3}\sqrt{t}})^{-(2n+\theta)},
\end{align*}
where the contour integral over $C_{2}$ only equals to the residue at $1$ for $\xi_{2}=-e^{\ii\xi_{3}\sqrt{t}}$ is outside of $C_{2}$.
Thus,
\begin{align*}
d_{n}^{\theta}(t)=&\frac{1}{\sqrt{2\pi}}\exp\Big\{\frac{(\theta-1)^2t}{8}\Big\}\int_{-\infty-\ii a}^{\infty-\ii a}
\exp\Big\{-\frac{\xi_{3}^2}{2}-\ii\xi_{3}(\frac{2n+\theta-3}{2})\sqrt{t}\Big\}\\
&e^{\ii\xi_{3}\sqrt{t}(2n-2+\theta)}(e^{\ii\xi_{3}\sqrt{t}}-1)\binom{2n-1+\theta}{n}(1+e^{\ii\xi_{3}\sqrt{t}})^{-(2n+\theta)}
d\xi_{3}\\
=&\binom{2n-1+\theta}{n}\exp\Big\{\frac{(\theta-1)^2t}{8}\Big\}\frac{1}{\sqrt{2\pi}}\int_{-\infty-\ii a}^{\infty-\ii a}\\
&\exp\Big\{-\frac{\xi_{3}^2}{2}+\ii\xi_{3}(\frac{2n+\theta-1}{2})\sqrt{t}\Big\}\frac{(e^{\ii\xi_{3}\sqrt{t}}-1)}{(1+e^{\ii\xi_{3}\sqrt{t}})^{(2n+\theta)}}d\xi_{3}
\end{align*}
Consider substitution $w=\xi_{3}\sqrt{t}$, then one will have
\begin{align*}
d_{n}^{\theta}(t)=&\binom{2n-1+\theta}{n}\frac{1}{2^{2n-1+\theta}}e^{\frac{(1-\theta)^2t}{8}}
\frac{1}{\sqrt{2\pi t}}\int_{-\infty-\ii A}^{\infty-\ii A}
e^{-\frac{w^2}{2t}}\frac{\ii\sin\frac{w}{2}}{(\cos\frac{w}{2})^{2n+\theta}}dw
\end{align*}
where $A=a\sqrt{t}>\log(3/2)$. But one can show that $A$ can be any positive number because for $w=R+\ii y,~y\in[A^{'},A]$
$$
\left|e^{-\frac{w^2}{2t}}\frac{\ii\sin\frac{w}{2}}{(\cos\frac{w}{2})^{2n+\theta}}\right|\leq M e^{-\frac{R^2}{2t}}.
$$
The usual argument for Fourier transform shows that one can remove the restriction for $A$.
\end{proof}

\section{Local Central Limit Theorem}

As we know that $\lim_{t\to0}tD_{t}=2$~(refer to \cite{MR2574323}). One needs to consider asymptotic behavior of $\p(D_{t}=\left[\frac{2+\sqrt{t}v}{t}\right])$ to establish central limit theorem.  However, the integral representation in Theorem \ref{asm} is not suitable for asymptotic analysis, so we derive an alternative integral representation through contour deformation.

\begin{thm}\label{arep}
For $n\geq0$, $d_{n}^{\theta}(t)$ has the following integral representation
\begin{align*}
d_{n}^{\theta}(t)=&\p(D_{t}=n)=\binom{2n-1+\theta}{n}\frac{c_{n}^{\theta}(t)}{2^{2n-1+\theta}} 
\end{align*}
where $C$ is a unit circle,
$$
c_{n}^{\theta}(t)=\prod_{k=1}^{\infty}[1-e^{\frac{-4k\pi^2}{t}}]e^{\frac{(1-\theta)^2t}{8}}
\frac{\ii}{2\sqrt{2\pi t}}\int_{C}(K_{t}(z)-K_{t}(-z))\phi_{t}(z)dz,
$$ 
$$
K_{t}(z)=e^{-\frac{(\pi+z)^2}{2t}}\frac{\sin(\frac{z+\pi}{2})}{\cos^{2n+\theta}(\frac{z+\pi}{2})}
$$
and
$$
\phi_{t}(z)=\prod_{k=1}^{\infty}\left[1+(-1)^{\theta-1}2\cosh\left(\frac{2\pi z}{t}\right)e^{-\frac{4\pi^2 k}{t}}+e^{-\frac{8\pi^2 k}{t}}\right].
$$
\end{thm}
\begin{proof}
 Let 
$$
f(w)=e^{-\frac{w^2}{2t}}\frac{\ii\sin\frac{w}{2}}{(\cos\frac{w}{2})^{2n+\theta}}
$$
One can easily show that $f(-w)=-f(w)$. Therefore,
\begin{align*}
\int_{-\infty-\ii A}^{\infty-\ii A}f(w)dw=\int_{\infty+\ii A}^{-\infty+\ii A}-f(-w)dw=\int_{\infty+\ii A}^{-\infty+\ii A}f(w)dw\end{align*}
Then
$$
2\int_{-\infty-\ii A}^{\infty-\ii A}f(w)dw=\int_{-\infty-\ii A}^{\infty-\ii A}f(w)dw+\int_{\infty+\ii A}^{-\infty+\ii A}f(w)dw=\int_{\mathcal{L}}f(w)dw
$$
where contour $\mathcal{L}$ are two parallel horizontal lines. One can easily deform the contour $\mathcal{L}$ to contours in Figure \ref{drec}.

\begin{figure}
  \centering
  \includegraphics[width=3.5in,height=3in]{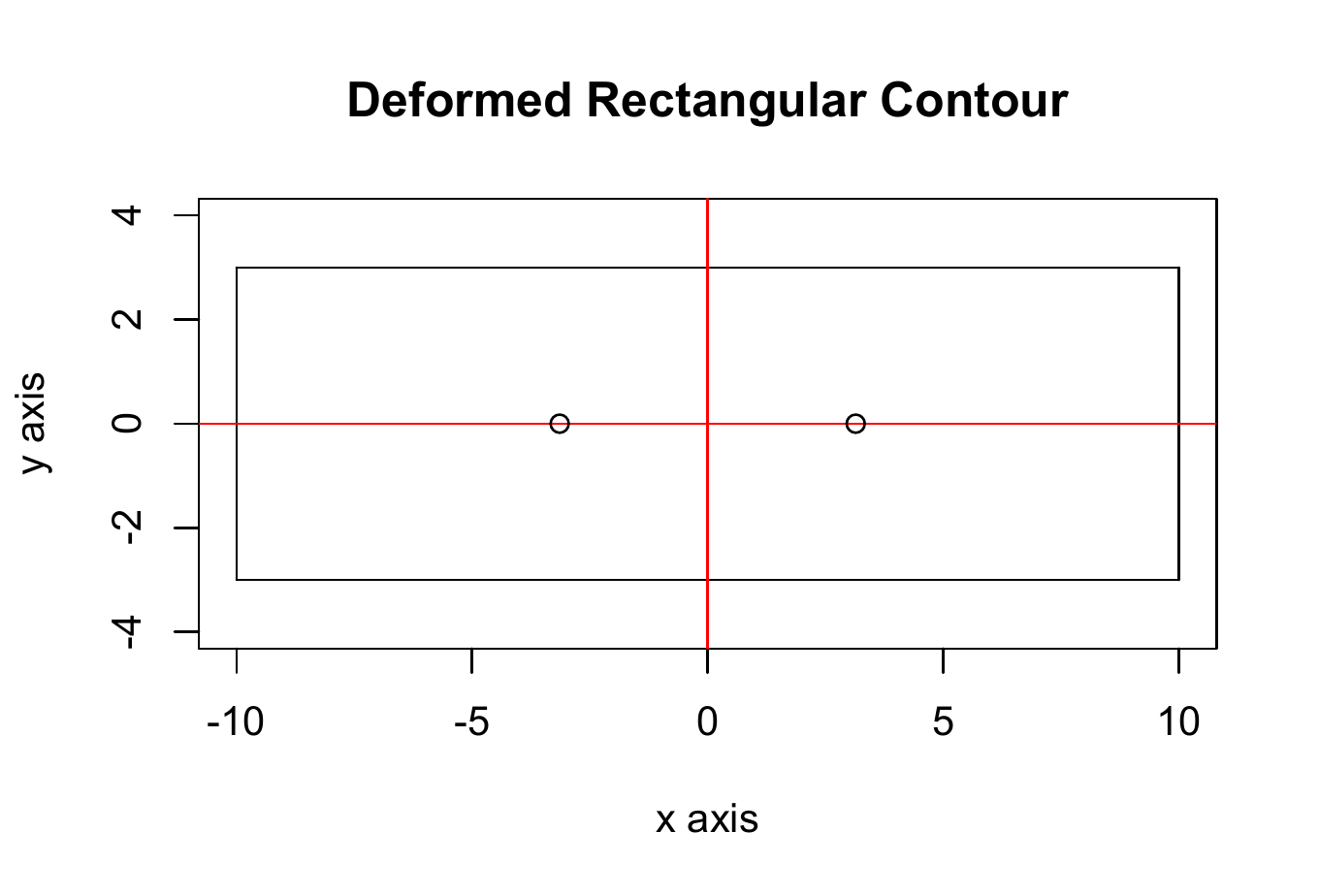}
  \caption{Deformed Rectangular Contour}\label{drec}
\end{figure}

 Even though $f(w)$ is not analytic on segments $[(4k+1)\pi,(4k+3)\pi],k\geq1$ for they are cut lines of $(\cos\frac{w}{2})^{2n+\theta}$. Due to symmetry of $f(w)$, their integrations on cut lines are perfectly cancelled. 
 
 Therefore, 
$$
\int_{-\infty-\ii A}^{\infty-\ii A}f(w)dw=\sum_{k\in\mathbb{Z}}\int_{C_{2k+1}}f(w)dw=\int_{C}\sum_{k\in\mathbb{Z}}f((2k+1)\pi+z)dz
$$
where $C_{2k+1}$ is chosen to be a circle centered at $(2k+1)\pi$ with radius $1$ and $C$ is a unit circle. 
Therefore,
\begin{align*}
&\sum_{k\in\mathbb{Z}}f((2k+1)\pi+z)=\sum_{k\in\mathbb{Z}}e^{-\frac{((2k+1)\pi+z)^2}{2t}}\frac{\ii\sin((2k+1)\pi/2+\frac{z}{2})}{\cos^{2n+\theta}((2k+1)\pi/2+\frac{z}{2})}\\
=&\sum_{k\in\mathbb{Z}}e^{-\frac{((2k+1)\pi+z)^2}{2t}}\frac{(-1)^k}{(-1)^{k\theta}}\frac{\ii\sin(\frac{z+\pi}{2})}{\cos^{2n+\theta}(\frac{z+\pi}{2})}\\
=&\left[\sum_{k\in\mathbb{Z}}(-1)^{k(\theta-1)}e^{-\frac{((2k+1)\pi+z)^2}{2t}}\right]\frac{\ii\sin(\frac{z+\pi}{2})}{\cos^{2n+\theta}(\frac{z+\pi}{2})}
\end{align*}
where $(-1)^{k(\theta-1)}$ is evaluated as principal branch. If we denote 
$$
h_{t}(z|\theta)=\sum_{k\in\mathbb{Z}}(-1)^{k(\theta-1)}e^{-\frac{[(2k+1)\pi+z]^2}{2t}},
$$
then
$$
\int_{-\infty-\ii A}^{\infty-\ii A}f(w)dw=\frac{\ii}{2}\int_{C}h_{t}(z|\theta)\frac{\sin(\frac{z+\pi}{2})}{\cos^{2n+\theta}(\frac{z+\pi}{2})}dz
$$
 The function $h_{t}(z|\theta)$ is indeed a double periodic function with period $4k\pi$ in $z$ and $2k$ in $\theta$.
It is actually the famous Jacobi theta function. In the following, we need the famous Jacobi triple product identity formula~(refer to \cite{abramowitz+stegun})
\begin{lemma}\label{tri}
For $|q|<1,x\neq 0$, we have
$$
\sum_{n\in\mathbb{Z}}q^{n^2}x^n=\prod_{n=1}^{\infty}(1+xq^{2n-1})(1+x^{-1}q^{2n-1})(1-q^{2n})
$$
\end{lemma}
Because one can rewrite $h_{t}(z|\theta)$ as
\begin{align*}
h_{t}(z|\theta)=&e^{-\frac{(\pi+z)^2}{2t}}\sum_{k\in\mathbb{Z}}[e^{-\frac{2\pi^2}{t}}]^{k^2}[(-1)^{\theta-1}e^{-\frac{2\pi(\pi+z)}{t}}]^k,
\end{align*}
where
$$
q=e^{-\frac{2\pi^2}{t}},\quad x=(-1)^{\theta-1}e^{-\frac{2\pi(\pi+z)}{t}}.
 $$
Then by Lemma \ref{tri}, we have
\begin{align*}
h_{t}(z|\theta)=&e^{-\frac{(\pi+z)^2}{2t}}\prod_{k=1}^{\infty}[1+(-1)^{\theta-1}e^{-\frac{2\pi(\pi+z)}{t}}e^{-\frac{2(2k-1)\pi^2}{t}}]\\
&[1+(-1)^{\theta-1}e^{\frac{2\pi(\pi+z)}{t}}e^{-\frac{2(2k-1)\pi^2}{t}}]\\
&[1-e^{-\frac{4k\pi^2}{t}}]\\
=&[e^{-\frac{(\pi+z)^2}{2t}}+(-1)^{\theta-1}e^{-\frac{(\pi-z)^2}{2t}}] \prod_{k=1}^{\infty}[1-e^{-\frac{4k\pi^2}{t}}] \\
&\prod_{k=1}^{\infty}[1+(-1)^{\theta-1}2\cosh(2\pi z/t)e^{-\frac{4k\pi^2}{t}}+e^{-\frac{8k\pi^2}{t}}]
\end{align*}
Note that 
$$
\sin\frac{\pi+z}{2}=\sin\frac{\pi-z}{2}, \quad -\cos\frac{\pi+z}{2}=\cos\frac{\pi-z}{2}
$$
then 
\begin{align*}
&[e^{-\frac{(\pi+z)^2}{2t}}+(-1)^{\theta-1}e^{-\frac{(\pi-z)^2}{2t}}]\frac{\sin(\frac{z+\pi}{2})}{\cos^{2n+\theta}(\frac{z+\pi}{2})}\\
=&e^{-\frac{(\pi+z)^2}{2t}}\frac{\sin(\frac{z+\pi}{2})}{\cos^{2n+\theta}(\frac{z+\pi}{2})}+(-1)^{\theta-1}e^{-\frac{(\pi-z)^2}{2t}}\frac{\sin(\frac{\pi-z}{2})}{(-1)^{2n+\theta}\cos^{2n+\theta}(\frac{\pi-z}{2})}\\
=&e^{-\frac{(\pi+z)^2}{2t}}\frac{\sin(\frac{z+\pi}{2})}{\cos^{2n+\theta}(\frac{z+\pi}{2})}-e^{-\frac{(\pi-z)^2}{2t}}\frac{\sin(\frac{\pi-z}{2})}{\cos^{2n+\theta}(\frac{\pi-z}{2})}
\end{align*}
If we denote 
$$
K_{t}(z)=e^{-\frac{(\pi+z)^2}{2t}}\frac{\sin(\frac{z+\pi}{2})}{\cos^{2n+\theta}(\frac{z+\pi}{2})},
$$
then our theorem is proved. 
\end{proof}

\subsection{Steep Descent Contour for Asymptotic Analysis} We are going to use steepest descent method to do asymptotic analysis, so we need to find out steep descent path for asymptotic analysis. 
If we denote
$$
g_{t}(z)=\left[e^{-\frac{(\pi+z)^2}{2t}}\frac{\sin(\frac{z+\pi}{2})}{\cos^{2n+\theta}(\frac{z+\pi}{2})}-e^{-\frac{(\pi-z)^2}{2t}}\frac{\sin(\frac{\pi-z}{2})}{\cos^{2n+\theta}(\frac{\pi-z}{2})}\right]\phi_{t}(z)
$$
one can easily show that $g_{t}(z)$ is an odd function, i.e. $g_{t}(-z)=-g_{t}(z)$. Note that $\phi_{t}(z)$ is an even function. Then
\begin{align*}
\int_{C}g_{t}(z)\phi_{t}(z)dz=&\int_{C_{+}}g_{t}(z)\phi_{t}(z)dz+\int_{C_{-}}g_{t}(z)\phi_{t}(z)dz\\
=&\int_{C_{+}}g_{t}(z)\phi_{t}(z)dz-\int_{C_{+}}g_{t}(-z)\phi_{t}(-z)dz\\
=&\int_{C_{+}}g_{t}(z)\phi_{t}(z)dz+\int_{C_{+}}g_{t}(z)\phi_{t}(z)dz\\
=&2\int_{C_{+}}g_{t}(z)\phi_{t}(z)dz
\end{align*}
where $C_{+}$ is the upper half circle and it has counter clockwise orientation. Hence
\begin{equation}
c_{n}^{\theta}(t)=\prod_{k=1}^{\infty}[1-e^{\frac{-4k\pi^2}{t}}]e^{\frac{(1-\theta)^2t}{8}}
\frac{\ii}{\sqrt{2\pi t}}\int_{C+}(K_{t}(z)-K_{t}(-z))\phi_{t}(z)dz. \label{eqq}
\end{equation}
Now we consider two paths~(Figure \ref{stt})
\begin{figure}
  \centering
  \includegraphics[width=3in,height=2.8in]{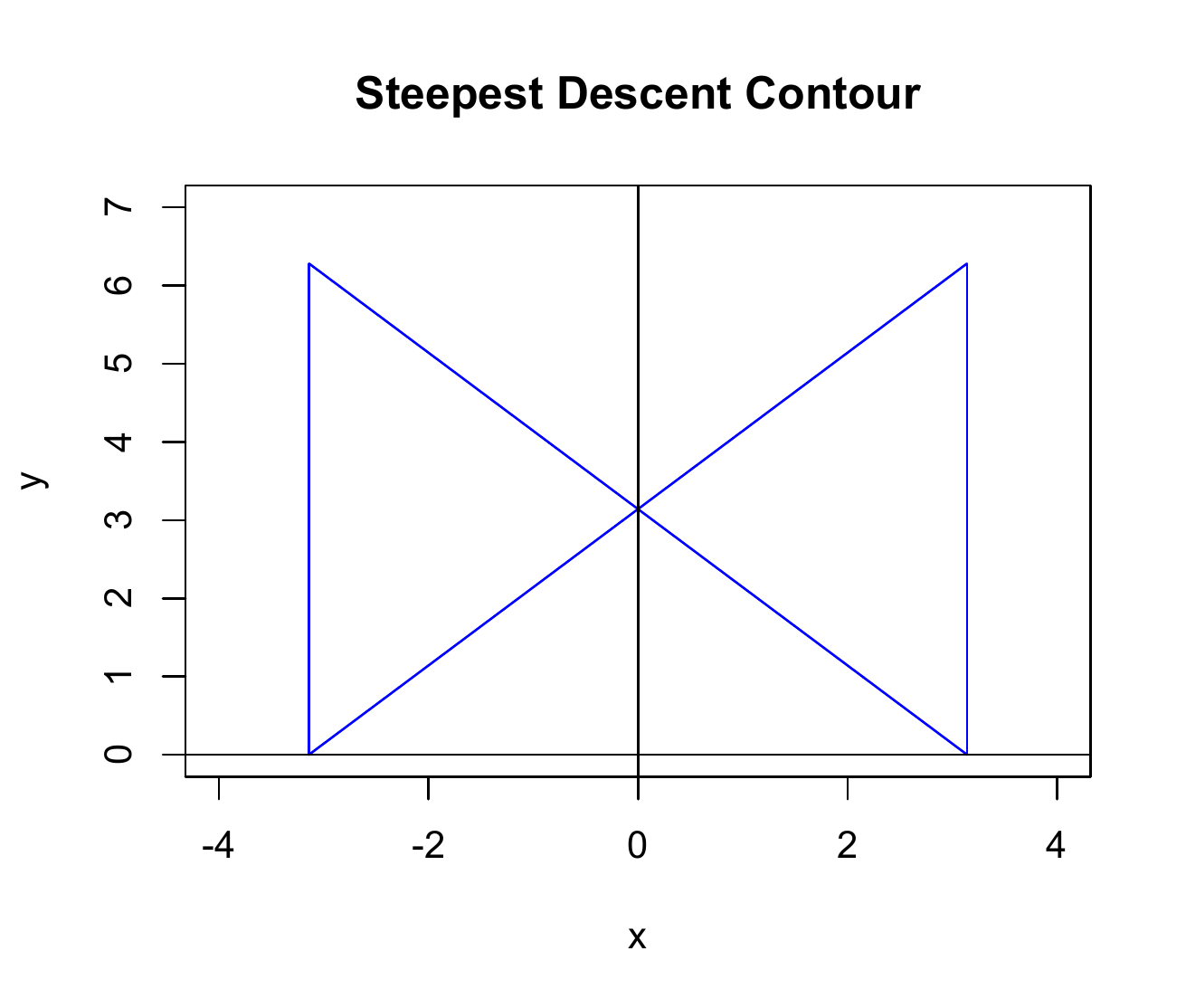}
  \caption{Steep Descent Contour}\label{stt}
\end{figure}
$$
L_{1}: z(y)=\begin{cases}
-\pi+y+\ii y, & 0\leq y\leq 2\pi\\
\pi+\ii y,& 2\pi\leq y\leq 0
\end{cases}
$$
$$
L_{2}: z(y)=\begin{cases}
\pi-y+\ii y, & 0\leq y\leq 2\pi\\
-\pi+\ii y,& 2\pi\leq y\leq 0
\end{cases}
$$
These two paths $L_{1}$ and $L_{2}$ serve as steep descent path for integral $\int_{C+}K_{t}(z)\phi_{t}(z)dz$ and $-\int_{C+}K_{t}(-z)\phi_{t}(z)dz$ respectively though neither of them are steepest descent contour. However, these two paths are good enough for our analysis. Indeed, we can rewrite $K_{t}(z)$ as 
$$
e^{-\frac{(z+\pi)^2}{2t}}\frac{\sin\frac{z+\pi}{2}}{\cos^{2n+\theta}(\frac{z+\pi}{2})}=\exp\left\{\frac{1}{t}\psi_{t}(\frac{z+\pi}{2})\right\}\sin\frac{z+\pi}{2}$$
where 
$$
\psi_{t}(z)=-2z^2-(2n+\theta)t\log\cos(z);
$$ 
similarly, 
$$
K_{t}(-z)=\exp\left\{\frac{1}{t}\psi_{t}(\frac{-z+\pi}{2})\right\}\sin\frac{-z+\pi}{2}.
$$
Since 
$
\psi_{t}^{'}(z)=0
$
has one solution $z=0$. Then $z=-\pi$ is a critical point of $\psi_{t}(\frac{z+\pi}{2})$, and $z=\pi$ is a critical point of $\psi_{t}(\frac{-z+\pi}{2})$. Moreover, 
because on $L_{1}$, 
$$
a(y)=\mathrm{Re}[\psi_{t}(\frac{z(y)+\pi}{2})]=-\frac{(2n+\theta)t}{2}\log\frac{\cosh y+\cos y}{2}
$$
one can easily check that 
$$
a^{'}(y)=-(2n+\theta)t\frac{\sinh y-\sin y}{\cosh y+\cos y}\leq0,\quad \forall 0\leq y\leq 2\pi.
$$
therefore $a(y)$ is decreasing as $y$ increases. Likewise, 
$$
b(y)=\mathrm{Re}[\psi_{t}(\frac{z(y)-\pi}{2})]=-\frac{(2n+\theta)t}{2}\log\frac{\cosh y+\cos y}{2}=a(y)
$$
and $b(y)$ is also decreasing as $y$ increases. So $\mathrm{Re}[\psi_{t}(\frac{z+\pi}{2})]$ and $\mathrm{Re}[\psi_{t}(\frac{-z+\pi}{2})]$ are both decreasing as $z$ is moving away from $-\pi$ along $L_{1}$ and from $\pi$ along $L_{2}$ respectively. Also one can eventually show that the integration over remaining vertical paths on $L_{1}$ and $L_{2}$ can be ignored. Thus, the major contribution is around $\pi$ and $-\pi$. So $L_{1}$ and $L_{2}$ serve as steep descent contour for $\int_{C+}K_{t}(z)\phi_{t}(z)dz$ and $-\int_{C+}K_{t}(-z)\phi_{t}(z)dz$ respectively. Also $L_{1}$ and $L_{2}$ has a one-to-one correspondence $w=-\overline{z}$

\subsection{Contour Deformation to Steep Descent Contour} To deform the upper half unit circle to steep descent contour, we can first deform the upper half unit circle to a big upper half circle with radius $\pi$. Though $K_{t}(z)$ and $K_{t}(-z)$ are not analytic on the $[-\pi,\pi]$, the deformation can still go through because the integrand is odd. Therefore,
\begin{align*}
&\int_{C+}(K_{t}(z)-K_{t}(-z))\phi_{t}(z)dz=-\int_{L_{1}}K_{t}(z)\phi_{t}(z)dz-\int_{L_{2}}K_{t}(-z)\phi_{t}(z)dz\\
=&-\int_{L_{1}}K_{t}(z)\phi_{t}(z)dz+\int_{L_{1}}K_{t}(\overline{z})\phi_{t}(\overline{z})d\overline{z}\\
=&-\int_{L_{1}}\exp\left\{\frac{1}{t}\psi_{t}(\frac{z+\pi}{2})\right\}\sin\frac{\pi+z}{2}\phi_{t}(z)dz+\int_{L_{1}}\exp\left\{\frac{1}{t}\psi_{t}(\frac{\overline{z}+\pi}{2})\right\}\sin\frac{\pi+\overline{z}}{2}\phi_{t}(\overline{z})d\overline{z}\end{align*}
where $L_{1}$ and $C_{+}$ has opposite orientation.
If we let $C_{t}(\pi)$ and $C_{t}(-\pi)$ be two discs with center $\pi$ and $-\pi$ and radius $t^{\frac{1}{4}}\log\frac{1}{t}$. If we denote $L_{1,t}$ be the segment within $C_{t}(-\pi)$ and $L_{1,t}^c$ be the segment outside $C_{t}(-\pi)$.
\begin{lemma}\label{es}
There exists $\delta>0$, such that $\forall 0<t<\delta$
$$
\left|\int_{L_{1,t}^c}K_{t}(z)\phi_{t}(z)dz\right|=\left|\int_{L_{1,t}^c}K_{t}(\overline{z})\phi_{t}(\overline{z})d\overline{z}\right|\leq M_{1} \exp\left\{-M_{2}\log^4\frac{1}{t}\right\}
$$
where $M_{1},M_{2}$ are two constants and independent of $t$. 
\end{lemma}

 One will see that the major contribution is from the integration over $L_{1,t}$. 
To pinpoint the above limits, we need to expand $\psi_{t}(\frac{z+\pi}{2})$ on $L_{1,t}$. Because 
\begin{align}
\psi_{t}(z)=&-2z^2-(2n+\theta)t\log\cos(z)
=\left[-2+\frac{(2n+\theta)t}{2}\right]z^2+(2n+\theta)t\sum_{n=2}^{\infty}(-1)^n\frac{2^{2n}(2^{2n}-1)B_{2n}}{2n(2n)!}z^{2n}\label{ee}
\end{align}
where $B_{2n}$ is Bernoulli number~(refer to \cite{abramowitz+stegun} for above expansion). We consider the following parametrizations:
$$
L_{1,t}:  z(y)=-\pi+t^{\frac{1}{4}}e^{\ii\frac{\pi}{4}} \sqrt{y},\quad 0\leq y\leq (\log\frac{1}{t})^2.
$$ 
Then one can show the following lemma. 
\begin{lemma}\label{con}
As $t\to0$, we have
\begin{align*}
-\int_{L_{1,t}}K_{t}(z)\phi_{t}(z)dz+\int_{L_{1,t}}K_{t}(\overline{z})\phi_{t}(\overline{z})d\overline{z}
\sim -\ii\frac{2\sqrt{6}\sqrt{2\pi}}{4}e^{-\frac{3v^3}{4}}\sqrt{t}. 
\end{align*}
\end{lemma} 

Thus, we have our main theorem.
\begin{thm}
Let $n=\left[\frac{2+\sqrt{t}v}{t}\right]$. As $t\to0$, we have
$$
\p\left(D_{t}=\left[\frac{2+\sqrt{t}v}{t}\right]\right)=d_{n}^{\theta}(t)\sim \sqrt{\frac{3}{2}}\frac{\sqrt{t}}{\sqrt{2\pi}} e^{-\frac{3v^3}{4}}.
$$
\end{thm}
\begin{proof}
By Theorem \ref{arep}, we know 
$$
d_{n}^{\theta}(t)=\p(D_{t}=n)=\binom{2n-1+\theta}{n}\frac{c_{n}^{\theta}(t)}{2^{2n-1+\theta}} 
$$
where 
$$
c_{n}^{\theta}(t)=\prod_{k=1}^{\infty}[1-e^{\frac{-4k\pi^2}{t}}]e^{\frac{(1-\theta)^2t}{8}}
\frac{\ii}{2\sqrt{2\pi t}}\int_{C}(K_{t}(z)-K_{t}(-z))\phi_{t}(z)dz.
$$ 
By Stirling's formula, we know as $n\to\infty$
$$
\binom{2n-1+\theta}{n}\frac{1}{2^{2n-1+\theta}} \sim \frac{1}{\sqrt{2\pi}}\sqrt{\frac{2}{n}}.
$$
Due to Lemma \ref{es} and Lemma \ref{con} and equation (\ref{eqq}), one can see that 
$$
\int_{C}(K_{t}(z)-K_{t}(-z))\phi_{t}(z)dz\sim -\frac{\sqrt{t}\ii}{4}2\sqrt{6}\sqrt{2\pi}e^{-\frac{3v^3}{4}}.
$$
Therefore, 
$$
d_{n}^{\theta}(t)\sim \sqrt{\frac{3}{2}}\frac{\sqrt{t}}{\sqrt{2\pi}} e^{-\frac{3v^2}{4}}.
$$
\end{proof}
\section{Proofs of Lemma \ref{es} and Lemma \ref{con}}

\subsection{Proof of Lemma \ref{es}}
\begin{proof}
The path $L_{1,t}^{c}$ has two parts
\begin{align*}
\mathrm{I}: &z(r)=-\pi+re^{\ii\frac{\pi}{4}}, t^{\frac{1}{4}}\log\frac{1}{t} \leq r\leq 2\pi,\\
\mathrm{II}:& z(y)=\pi+iy, 2\pi\leq y\leq 0.
\end{align*}
On $\mathrm{II}$,  
$$
K_{t}(z(y))\phi_{t}(z(y))=-\ii\exp\left\{\frac{1}{t}\psi_{t}(z(y))\right\}\sinh\frac{y}{2}\phi_{t}(z(y)).
$$
where $\psi_{t}(z(y))=\frac{y^2}{2}-(2n+\theta)t\log\cosh\frac{y}{2}$. Then
\begin{align*}
&-\int_{\mathrm{II}}K_{t}(z)\phi_{t}(z)dz+\int_{\mathrm{II}}K_{t}(\overline{z})\phi_{t}(\overline{z})d\overline{z}\\
=&\int_{0}^{2\pi}\exp\left\{\frac{1}{t}\psi_{t}(z(y))\right\}\sinh\frac{y}{2}\phi_{t}(z(y))dy-\int_{0}^{2\pi}\exp\left\{\frac{1}{t}\psi_{t}(z(y))\right\}\sinh\frac{y}{2}\phi_{t}(\overline{z(y)})dy\\
=&\int_{0}^{2\pi}\exp\left\{\frac{1}{t}\psi_{t}(z(y))\right\}\sinh\frac{y}{2}\mathrm{Im}[\phi_{t}(z(y))]dy\end{align*}
By the monotonicity of $\psi_{t}(z(y))$, we know that  
\begin{align*}
&\frac{1}{t}\psi_{t}(z(y))\leq \frac{1}{t}\max\left\{\psi_{t}(z(2\pi)),\psi_{t}(z(0))\right\}\\
\leq& \max\left\{\frac{2\pi^2}{t}-(2n+\theta)\log(1+\frac{\pi^2}{2}),0\right\}\leq \max\left\{\frac{2\pi^2}{t}-\frac{(2n+\theta)t}{2}\frac{\pi^2}{t},0\right\}
\end{align*}
and 
$
\phi_{t}(z(y))=1+(-1)^{\theta-1}2\cosh\frac{2\pi z(y)}{t}+o(e^{-\frac{4\pi^2}{t}}).
$
Thus,
$$
|\mathrm{Im}(\phi_{t}(z(y)))|=\left|\mathrm{Im}\left[(-1)^{\theta-1}2\cosh(\frac{2\pi z(y)}{t})e^{-\frac{4\pi^2}{t}}\right]\right|\leq M_{1} e^{-\frac{2\pi^2}{t}}
$$
So 
$$
\exp\left\{\frac{1}{t}\psi_{t}(z(y))\right\}\mathrm{Im}[\phi_{t}(z(y))]\leq \exp\left\{-\frac{(2n+\theta)t}{2}\frac{\pi^2}{t}\right\}\leq \exp\left\{-M_{2}\frac{\pi^2}{t}\right\}
$$
and 
$$
\left|-\int_{\mathrm{II}}K_{t}(z)\phi_{t}(z)dz+\int_{\mathrm{II}}K_{t}(\overline{z})\phi_{t}(\overline{z})d\overline{z}\right|\leq M_{1}\left[\int_{0}^{2\pi}\sinh(\frac{y}{2})dy\right] e^{-M_{2}\frac{\pi^2}{t}}.
$$
On part $\mathrm{I}$, we know $|K_{t}(z(r))|=\exp\{\frac{1}{t}\mathrm{Re}[\psi_{t}(z(r))]\}|\sin\frac{r}{2}e^{\ii\frac{\pi}{4}}|$, $|\sin\frac{r}{2}e^{\ii\frac{\pi}{4}}|\leq e^{\frac{\pi}{\sqrt{2}}}$ and $\mathrm{Re}[\psi_{t}(z(r))]$ is decreasing in $r$, then
$$
|K_{t}(z(r))|\leq \exp\left\{\frac{1}{t}\mathrm{Re}[\psi_{t}(z(t^{\frac{1}{4}}\log\frac{1}{t}))]\right\}e^{\frac{\pi}{\sqrt{2}}}.
$$
It is not difficult to show that as $t\to0$
$$
\frac{1}{t}\mathrm{Re}[\psi_{t}(z(t^{\frac{1}{4}}\log\frac{1}{t}))]=-\frac{(2n+\theta)}{2}\log\frac{\cosh(t^{\frac{1}{4}}\log\frac{1}{t})) +\cos(t^{\frac{1}{4}}\log\frac{1}{t})) }{2}\sim -\frac{1}{12}\log^4\frac{1}{t}.
$$
Therefore, there exists $\delta>0$ such that 
$$
\frac{1}{t}\mathrm{Re}[\psi_{t}(z(t^{\frac{1}{4}}\log\frac{1}{t}))]\leq-\frac{1}{24}\log^4\frac{1}{4},\quad \forall 0<t<\delta.
$$
Moreover, $|\phi_{t}(z(r))|\leq \exp\{\sum_{k=1}^{\infty}\log(1+2\cosh\frac{2\pi z(r)}{t}e^{-\frac{4\pi^2k}{t}}+e^{-\frac{8\pi^2 k}{t}})\}$, then by inequality $\log(1+x)<x,\forall x>0$, one can show that 
$$
|\phi_{t}(z(r))|\leq M_{1}
$$
Thus,
\begin{align*}
&\left|-\int_{\mathrm{I}}K_{t}(z)\phi_{t}(z)dz+\int_{\mathrm{I}}K_{t}(\overline{z})\phi_{t}(\overline{z})d\overline{z}\right|\\
\leq& 4\pi\sqrt{2}|K_{t}(z)\phi_{t}(z)|\leq M_{1}\exp\left\{-\frac{1}{24}\log^4\frac{1}{t}\right\}. 
\end{align*}
So 
$$
\left|-\int_{L_{1,t}^c}K_{t}(z)\phi_{t}(z)dz+\int_{L_{1,t}^c}K_{t}(\overline{z})\phi_{t}(\overline{z})d\overline{z}\right|\leq M_{1}\exp\left\{-M_{2}\log^4\frac{1}{t}\right\},\quad M_{1},M_2>0.
$$
\end{proof}

\subsection {Proof of Lemma \ref{con}}
\begin{proof}

Consider the following parametrization:
$$
L_{1,t}:  z(y)=-\pi+t^{\frac{1}{4}}e^{\ii\frac{\pi}{4}} \sqrt{y},\quad 0\leq y\leq (\log\frac{1}{t})^2.
$$ 
Then by the Taylor expansion of $\psi_{t}(z)$ in (\ref{ee}), one can show that there exists $\delta>0$ such that 
$$
\frac{1}{t}\psi_{t}(\frac{z(y)+\pi}{2})=\frac{(2n+\theta)t-4}{8\sqrt{t}}y\ii-\frac{(2n+\theta)t}{12\times 2^4}y^2+\Delta_{1,t}(y)
$$
$$
\frac{1}{t}\psi_{t}(\frac{\pi+\overline{z(y)}}{2})=-\frac{(2n+\theta)t-4}{8\sqrt{t}}y\ii-\frac{(2n+\theta)t}{12\times 2^4}y^2+\Delta_{2,t}(y)
$$
where $|\Delta_{1,t}(y)|,|\Delta_{2,t}(y)|\leq M \sqrt{t}\log^6\frac{1}{t},\forall t<\delta,$ and $M$ is independent of $t$.  Thus, if we denote 
$$
\epsilon_{1,t}(y)=\phi_{t}(z(y))\exp\{\Delta_{1,t}(y)\}\frac{\sin(t^{\frac{1}{4}}\sqrt{y}e^{\frac{\pi}{4}\ii})}{2\sqrt{y}}t^{\frac{1}{4}}e^{\frac{\pi}{4}\ii}
$$
$$
\epsilon_{2,t}(y)=\phi_{t}(\overline{z(y)})\exp\{\Delta_{2,t}(y)\}\frac{\sin(t^{\frac{1}{4}}\sqrt{y}e^{-\frac{\pi}{4}\ii})}{2\sqrt{y}}t^{\frac{1}{4}}e^{-\frac{\pi}{4}\ii}
$$
then it is not difficult to show that as $t\to0$ the following results hold uniformly for $0<y<\log^2\frac{1}{t}$
\begin{equation}
\epsilon_{1,t}(y)\sim \frac{\sqrt{t}\ii}{4},\quad \epsilon_{2,t}(y)\sim -\frac{\sqrt{t}\ii}{4},\quad \phi_{t}(z(y))\sim 1,\quad \phi_{t}(\overline{z(y)})\label{eqv}
\sim 1.
\end{equation}
Therefore,
\begin{align*}
&-\int_{L_{1,t}}K_{t}(z)\phi_{t}(z)dz+\int_{L_{1,t}}K_{t}(\overline{z})\phi_{t}(\overline{z})d\overline{z}\\
=& -\int_{0}^{\log^2\frac{1}{t}}\exp\left\{\frac{(2n+\theta)t-4}{8\sqrt{t}}y\ii-\frac{(2n+\theta)t}{12\times 2^4}y^2\right\}\epsilon_{1,t}(y)dy\\
&+\int_{0}^{\log^2\frac{1}{t}}\exp\left\{-\frac{(2n+\theta)t-4}{8\sqrt{t}}y\ii-\frac{(2n+\theta)t}{12\times 2^4}y^2\right\}\epsilon_{2,t}(y)dy\\
=& -\frac{\sqrt{t}\ii}{4}\int_{0}^{\log^2\frac{1}{t}}\exp\left\{\frac{(2n+\theta)t-4}{8\sqrt{t}}y\ii-\frac{(2n+\theta)t}{12\times 2^4}y^2\right\}dy\\
&-\frac{\sqrt{t}\ii}{4}\int_{0}^{\log^2\frac{1}{t}}\exp\left\{-\frac{(2n+\theta)t-4}{8\sqrt{t}}y\ii-\frac{(2n+\theta)t}{12\times 2^4}y^2\right\}dy\\
&-\int_{0}^{\log^2\frac{1}{t}}\exp\left\{\frac{(2n+\theta)t-4}{8\sqrt{t}}y\ii-\frac{(2n+\theta)t}{12\times 2^4}y^2\right\}\left[\epsilon_{1,t}(y)-\frac{\sqrt{t}\ii}{4}\right]dy\\
&+\int_{0}^{\log^2\frac{1}{t}}\exp\left\{-\frac{(2n+\theta)t-4}{8\sqrt{t}}y\ii-\frac{(2n+\theta)t}{12\times 2^4}y^2\right\}\left[\epsilon_{2,t}(y)+\frac{\sqrt{t}\ii}{4}\right]dy\\
=&-\frac{\sqrt{t}\ii}{4}\int_{-\infty}^{\infty}\exp\left\{\frac{(2n+\theta)t-4}{8\sqrt{t}}y\ii-\frac{(2n+\theta)t}{12\times 2^4}y^2\right\}dy+R(t)
\end{align*}
where 
\begin{align*}
R(t)=&-\frac{\sqrt{t}\ii}{4}\int_{\log^2\frac{1}{t}}^{\infty}\exp\left\{\frac{(2n+\theta)t-4}{8\sqrt{t}}y\ii-\frac{(2n+\theta)t}{12\times 2^4}y^2\right\}dy\\
&-\frac{\sqrt{t}\ii}{4}\int_{\log^2\frac{1}{t}}^{\infty}\exp\left\{-\frac{(2n+\theta)t-4}{8\sqrt{t}}y\ii-\frac{(2n+\theta)t}{12\times 2^4}y^2\right\}dy\\
&-\int_{0}^{\log^2\frac{1}{t}}\exp\left\{\frac{(2n+\theta)t-4}{8\sqrt{t}}y\ii-\frac{(2n+\theta)t}{12\times 2^4}y^2\right\}\left[\epsilon_{1,t}(y)-\frac{\sqrt{t}\ii}{4}\right]dy\\
&+\int_{0}^{\log^2\frac{1}{t}}\exp\left\{-\frac{(2n+\theta)t-4}{8\sqrt{t}}y\ii-\frac{(2n+\theta)t}{12\times 2^4}y^2\right\}\left[\epsilon_{2,t}(y)+\frac{\sqrt{t}\ii}{4}\right]dy
\end{align*}
One can easily show that 
\begin{align*}
|R(t)|\leq& M_{1}\sqrt{t}\exp\{-M_{2}\log^4\frac{1}{t}\}+ \frac{\sqrt{t}}{4}\max_{0<y<\log^2\frac{1}{t}}|\frac{\epsilon_{1,t}(y)}{\sqrt{t}\ii/4}-1|+ \frac{\sqrt{t}}{4}\max_{0<y<\log^2\frac{1}{t}}|\frac{\epsilon_{2,t}(y)}{\sqrt{t}\ii/4}+1|. 
\end{align*}
Due to uniform convergence in (\ref{eqv}), we know $\lim_{t\to0}\frac{R(t)}{\sqrt{t}}=0$.

Because $n=[\frac{2+\sqrt{t}v}{t}]$, then $\lim_{t\to0}\frac{(2n+\theta)t-4}{8\sqrt{t}}=\frac{v}{4}$ and $\lim_{t\to0}(2n+\theta)t=4$. By Lebesgue's dominant convergent theorem, we can show that 
\begin{align*}
&\lim_{t\to0}\int_{-\infty}^{\infty}\exp\left\{\frac{(2n+\theta)t-4}{8\sqrt{t}}y\ii-\frac{(2n+\theta)t}{12\times 2^4}y^2\right\}dy
=\int_{-\infty}^{\infty}\exp\left\{v\frac{y}{4}\ii-\frac{1}{3\times 2^4}y^2\right\}dy
\end{align*}
Then by Fourier transformation, we know 
$$
\int_{-\infty}^{\infty}\exp\left\{v\frac{y}{4}\ii-\frac{1}{3\times 2^4}y^2\right\}dy=2\sqrt{6}\sqrt{2\pi}e^{-\frac{3v^3}{4}}. 
$$
So 
$$
-\int_{L_{1,t}}K_{t}(z)\phi_{t}(z)dz+\int_{L_{1,t}}K_{t}(\overline{z})\phi_{t}(\overline{z})d\overline{z}\sim -\ii\frac{ 2\sqrt{6}\sqrt{2\pi}}{4}e^{-\frac{3v^3}{4}}\sqrt{t}. 
$$
\end{proof}

\section{Conclusion Remarks}
The integral representation of probabilities $d_{n}^{\theta}(t)$ has other possible applications such as large deviations and moderate deviations for $D_{t}$ at small time regime. In proof of local central limit theorem, we use a steep descent path, which is actually not the steepest descent path. The steepest descent path is actually given by $\mathrm{Im}(\psi_{t}(z))=0$ otherwise the oscillation of the imaginary part will produce many cancellations. It will eventually prevent us from getting precise estimation.  But for the proof of local central limit theorem, our path is good enough. If one can carefully handle (\ref{eqv}), it is possible to get some asymptotic expansion.

 \bibliography{BS}{}
\bibliographystyle{plain} 
 
\end{document}